\documentclass[11pt]{article}%
\usepackage{amssymb}
\usepackage{amsfonts}
\usepackage{amsmath}
\usepackage{amssymb}
\usepackage{color}
\usepackage{graphicx}
\usepackage{caption}%
\usepackage{hyperref}
\usepackage{float}
\usepackage{subcaption}

\usepackage{tabularx}

\providecommand{\U}[1]{\protect\rule{.1in}{.1in}}
\newtheorem{theorem}{Theorem}

\newtheorem{remark}[theorem]{Remark}
\newtheorem{definition}[theorem]{Definition}
\newenvironment{proof}[1][Proof]{\noindent \textbf{#1.} }{\hspace{16cm} \rule{0.5em}{0.5em}}
\begin{document}
	
	\title{\textbf{The restricted minimum density power divergence estimator for non-destructive one-shot device testing the under step-stress model with exponential lifetimes  }}
	\author{Narayanaswamy Balakrishnan,  María Jaenada and Leandro Pardo}
	\date{ }
	\maketitle
	
	\begin{abstract}
	
		One-shot devices data represent an extreme case of interval censoring.
		Some kind of one-shot units do not get destroyed when tested, and so, survival units can continue within the test providing extra information about their lifetime.
		Moreover, one-shot devices may last for long times under normal operating conditions, and so accelerated life tests (ALTs) may be used for inference.
		ALTs relate the lifetime distribution of an unit with the stress level at which it is tested via log-linear relationship.
		Then, mean lifetime of the devices are reduced during the test by increasing the stress level and inference results on increased stress levels can be easily extrapolated to normal operating conditions.
		In particular, the step-stress ALT model increases the stress level at pre-fixed times gradually during the life-testing experiment, which may be specially advantageous for non-destructive one-shot devices.
		However, when the number of units under test are few, outlying data may greatly influence the parameter estimation.
		In this paper, we develop robust restricted estimators based on the density power divergence (DPD) under linearly restricted subspaces, for non-destructive one-shot devices under the step-stress ALTs with exponential lifetime distributions. We theoretically study the asymptotic and robustness properties of the restricted estimators and we empirically illustrate such properties through a simulation study.
	\end{abstract}

	\section{Introduction}
	
	One-shot devices, an extreme case of interval censoring, play an important role in survival analysis. One-shot devices can be only tested at some discrete inspections times, so we can only know if a test unit have failed or not at certain fixed times.
	Real-life one-shot devices usually have large mean lifetimes under normal operating conditions, and so accelerated life tests (ALTs) plans may be useful to infer on their reliability. 
	ALTs plans assume that the mean lifetime of the devices is related to the stress level at which units are tested, and therefore they
	accelerate the time to failure by increasing the stress level.
	
	Generally, one-shot devices are assumed to get destroyed when tested, and so one-shot data are right and left censoring. 
	However, the non-destructiveness assumption may not be necessary in many practical situations. In this paper, we focus on these non-destructive one-shot devices and we study some inference methods for analyzing their lifetime characteristics.
	The non-destructiveness condition allows surviving units to continue in the experiment, providing extra information about their lifetime distribution.
	
	In this context,  step-stress ALTs, which increase the stress level progressively during the experiment  at certain pre-specified times (known as times of stress change), make the best use of the non-destructive devices under tests.
	Here, we assume that the lifetime distribution of the one-shot device at one  stress level is related to the distribution at preceding stress levels by
	assuming the residual life of the device depends only on the cumulative exposure it had experienced,	with no memory of how this exposure was accumulated.
	We consider a multiple step-stress ALT with $k$ ordered stress levels, $x_1< x_2 < \dots < x_k$ and  their corresponding times of stress change $\tau_1 < \tau_2 \dots < \tau_k.$ 
	We assume the lifetimes of one-shot devices follows an exponential distribution, which is widely used as a simple lifetime model in engineering and physical sciences. 
	The cumulative exposure model describes the lifetime distribution of a device as
	 \begin{equation}\label{eq:distributionT}
	 	G_T(t) = 
	 	\begin{cases}
	 		G_1(t) = 1-e^{-\lambda_1 t}, & 0<t< \tau_1\\
	 		G_2\left(t + a_1 -\tau_1 \right)  = 1-e^{-\lambda_2(t + a_1 -\tau_1)}, & \tau_1 \leq t < \tau_2 \\
	 		\vdots  & \vdots \\
	 		G_k\left(t + a_{k-1} -\tau_{k-1} \right)  =  1-e^{-\lambda_k(t + a_{k-1} -\tau_{k-1})}, & \tau_{k-1} \leq t < \infty ,\\
	 	\end{cases}
	 \end{equation}
	 with 
	 \begin{equation}\label{eq:ai}
	 	a_{i-1} = \frac{ \sum_{l=1}^{i-1}\left(\tau_l-\tau_{l-1}\right)\lambda_l}{\lambda_i}, \hspace{0.3cm} i=1,...,k-1.
	 \end{equation}
	 and
	 \begin{equation} \label{eq:loglinear}
	 	\lambda_i(\boldsymbol{\theta}) = \theta_0 \exp(\theta_1 x_i), \hspace{0.3cm} i = 1,..,k,
	 \end{equation} 
	 where $\boldsymbol{\theta} = (\theta_0, \theta_1) \in \mathbb{R}^+ \times \mathbb{R} = \Theta$ is an unknown parameter vector of the model. 
	 The log-linear relation in (\ref{eq:loglinear}) is frequently assumed in accelerated life test models, as it can be shown to be equivalent to  the well-known inverse power law model or the Arrhenius reaction rate model.
	 
	Now, let consider a grid of inspection times, $t_1 < t_2 < \dots < t_L$, containing all times of stress change. The probability of failure within the interval $(t_{j-1}, t_{j}]$ is given by
	\begin{equation} \label{eq:th.prob}
		\pi_j(\boldsymbol{\theta}) =  G_T(t_j) - G_T(t_{j-1}), \hspace{0.3cm} j = 1,..,L,
	\end{equation}
	and the probability of survival at the end of the experiment is $\pi_{L+1}(\boldsymbol{\theta}) = 1 - G_T(t_L).$ 
	
 	Classical inferential methods for one-shot are based on the maximum likelihood estimator (MLE), which is very efficient by it lacks of robustness.
 	To overcome the robustness drawback, Balakrishnan et al. (2022) proposed robust estimators for one-shot devices based on the popular density power divergence (DPD) (Basu et al. 1998) under exponential lifetimes.
 	They developed minimum DPD estimators (MDPPE) as well as Wald-type test based on them, and studied theoretically and empirically their asymptotic and robustness properties.
 	
 	On the other hand, some inferential procedures, as Rao-type tests are based on restricted estimators.
 	 Basu et al. (2018) developed robust restricted estimators based on the DPD for general statistical models, and derived their asymptotic distribution and robustness properties. Jaenada et al. (2022) extended the theory using the Rényi pseudistance and developed some testing procedures based on the restricted estimators.
 	
 	
 	In this paper, we develop restricted MDPDE under linearly constrained subspaces for non-destructive one-shot devices tested under step-stress ALT.  In Section 2 we define the restricted MDPDE, and we state its asymptotic distribution. Section 3 theoretically analyzes the robustness of the restricted MDPDEs through its Influence Function (IF).
 	Finally, in Section 4 a simulation study is carried out to evaluate the performance of the proposed estimators under different scenarios of contamination.
	
	\section{Minimum density power divergence estimator and Restricted minimum density power divergence estimator \label{sec:modelformulation}}
	
	The density DPD family represents a rich class of density based divergences.
	It is indexed by a tuning parameter $\beta \geq 0$ controlling the trade-off between robustness and efficiency. 
	Let consider $(n_1,...,n_{L+1})$ a sample of one-shot data. The empirical probability vector of a multinomial model can be defined as 
	\begin{equation} \label{eq:emp.prob}
		\widehat{\boldsymbol{p}} = (\frac{n_1}{N},...,\frac{n_{L+1}}{N}).
	\end{equation}
	For step-stress ALT for one-shot devices under exponential lifetime distributions, the DPD between the 
	the empirical and theoretical probability vectors, defined in (\ref{eq:th.prob}) and (\ref{eq:emp.prob}), respectively,  is given by
	\begin{equation}\label{eq:DPDloss}
		d_{\beta}\left( \widehat{\boldsymbol{p}},\boldsymbol{\pi}\left(\boldsymbol{\theta}\right)\right)   = \sum_{j=1}^{L+1} \left(\pi_j(\boldsymbol{\theta})^{1+\beta} -\left( 1+\frac{1}{\beta}\right) \widehat{p}_j\pi_j(\boldsymbol{\theta})^{\beta}  +\frac{1}{\beta} \widehat{p}_j^{\beta+1} \right),
	\end{equation}
	From the above, we define the MDPDE for the step-stress ALT model with one-shot devices as
	\begin{equation}
		\boldsymbol{\widehat{\theta}}^{\beta} = \left(\widehat{\theta}_0^{\beta},\widehat{\theta}_{1}^{\beta}\right) = \operatorname{arg} \operatorname{min}_{\boldsymbol{\theta}  \in \Theta} d_{\beta} \left( \widehat{\boldsymbol{p}},\boldsymbol{\pi}\left(\boldsymbol{\theta}\right)\right).
	\end{equation}
	Note that, at $\beta = 0$, the DPD coincides with the Kullback-Leibler divergence and so the MDPDE for $\beta=0$ coincides with the MLE.
	
	In many practical situations it may be of interest to reduce the parameter space to values of $\boldsymbol{\theta}$ satisfying a linear constraint of the form
	\begin{equation}\label{eq:restriction}
		g(\boldsymbol{\theta}) = \boldsymbol{m}^T\boldsymbol{\theta}-d = 0,
	\end{equation}
	with $\boldsymbol{m} = (m_0,m_1)^T \in \mathbb{R}^2$ and $d\in \mathbb{R}.$ 
		Accordingly, the restricted MDPDE, $\widetilde{\boldsymbol{\theta}}^\beta$ is defined by
		\begin{equation}\label{RMDPDE}
			\widetilde{\boldsymbol{\theta}}^\beta = \operatorname{arg} \operatorname{min}_{\boldsymbol{\theta \in \Theta_0}} d_\beta(\widehat{\boldsymbol{p}}, \boldsymbol{\pi}(\boldsymbol{\theta})).
		\end{equation}
		
	Since the restricted MDPPE is a constrained minimum, its estimating equations can be written in terms of Lagrange multipliers. That is,
	the MDPDE restricted to the the linear constraint (\ref{eq:restriction}), $\widetilde{\boldsymbol{\theta}}^\beta,$ must satisfy the restricted equations
	\begin{equation}\label{eq:estimatingRMDPDE}
		\boldsymbol{W}^T\boldsymbol{D}_{\boldsymbol{\pi}(\widetilde{\boldsymbol{\theta}}^\beta)}^{\beta-1}\left( \widehat{\boldsymbol{p}}- \boldsymbol{\pi}(\widetilde{\boldsymbol{\theta}}^\beta)\right)+ \boldsymbol{m}\widetilde{\boldsymbol{\lambda}} = \boldsymbol{0}_2,
	\end{equation}
	for some vector $\widetilde{\boldsymbol{\lambda}}$ of Lagrangian multipliers,
	where $\boldsymbol{0}_2$ is the 2-dimensional null vector, $\boldsymbol{D}_{\boldsymbol{\pi}(\boldsymbol{\theta})}$ denotes a $(L+1)\times(L+1)$ diagonal matrix with diagonal entries $\pi_j(\boldsymbol{\theta}),$ $j=1,...,L+1,$ and $\boldsymbol{W}$ is a $(L+1) \times 2$ matrix with rows
	$
	\boldsymbol{w}_j =  \boldsymbol{z}_j-\boldsymbol{z}_{j-1},
	$
	where
	\begin{align}
		\label{eq:zj} \boldsymbol{z}_j  &= g_T(t_j)\begin{pmatrix}
			\frac{t_j+a_{i-1}-\tau_{i-1}}{\theta_0}\\
			(t_j+a_{i-1}-\tau_{i-1})x_i + a_{i-1}^\ast
		\end{pmatrix}, \hspace{0.3cm} j = 1,...,L,\\ 
		a_{i-1}^\ast &= \frac{1}{\lambda_{i}} \sum_{l=1}^{i-1}\lambda_l\left(\tau_l-\tau_{l-1}\right)(-x_{i}+x_l) , \hspace{0.3cm} i=2,..,k,\label{aast}
	\end{align}
	$ \boldsymbol{z}_{-1} = \boldsymbol{z}_{L+1} = \boldsymbol{0}$ and $i$ is the stress level at which the units are tested after the $j-$th inspection time.
	
	The next theorem states the asymptotic distribution of the restricted MDPDE for non-destructive one-shot devices under the step-stress ALT model.
	
	\begin{theorem}
		Let $\boldsymbol{\theta}_0$ be the true value of the parameter $\boldsymbol{\theta}$ and assume that $g(\boldsymbol{\theta}_0) = 0$ with $g(\cdot)$ defined in (\ref{eq:restriction}). 
		The asymptotic distribution of the restricted MDPDE for the step-stress ALT model under exponential lifetimes, $\tilde{\boldsymbol{\theta}}^\beta,$ obtained under the constraint $g(\boldsymbol{\theta})=0,$ is given by
		$$\sqrt{N}\left(\widetilde{\boldsymbol{\theta}}^\beta-\boldsymbol{\theta}_0 \right) \xrightarrow[N \rightarrow \infty]{L}\mathcal{N}\left(\boldsymbol{0}, \Sigma_\beta(\boldsymbol{\theta}_0)\right)$$
		where 
		\begin{equation}\label{eq:SigmaPQmatrices}
			\begin{aligned}
				\Sigma_\beta(\boldsymbol{\theta}_0) &= \boldsymbol{P}_\beta(\boldsymbol{\theta}_0) \boldsymbol{K}_\beta(\boldsymbol{\theta}_0) \boldsymbol{P}_\beta(\boldsymbol{\theta}_0),\\
				\boldsymbol{P}_\beta(\boldsymbol{\theta}_0) &=  \boldsymbol{J}_\beta(\boldsymbol{\theta}_0)^{-1}- \boldsymbol{Q}_\beta(\boldsymbol{\theta}_0)\boldsymbol{m}^T \boldsymbol{J}_\beta(\boldsymbol{\theta}_0)^{-1},\\
				\boldsymbol{Q}_\beta(\boldsymbol{\theta}_0) &=  \boldsymbol{J}_\beta(\boldsymbol{\theta}_0)^{-1}\boldsymbol{m}( \boldsymbol{m}^T \boldsymbol{J}_\beta(\boldsymbol{\theta}_0)^{-1}\boldsymbol{m})^{-1},
			\end{aligned}
		\end{equation} 
		with \begin{equation} \label{eq:JK}
			\boldsymbol{J}_\beta(\boldsymbol{\theta}_0) = \boldsymbol{W}^T D_{\boldsymbol{\pi}(\boldsymbol{\theta_0})}^{\beta-1} \boldsymbol{W},
			\hspace{0.3cm}  \hspace{0.3cm}
			\boldsymbol{K}_\beta(\boldsymbol{\theta}_0) = \boldsymbol{W}^T \left(D_{\boldsymbol{\pi}(\boldsymbol{\theta_0})}^{2\beta-1}-\boldsymbol{\pi}(\boldsymbol{\theta}_0)^{\beta}\boldsymbol{\pi}(\boldsymbol{\theta}_0)^{\beta T}\right) \boldsymbol{W},
		\end{equation}
		$D_{\boldsymbol{\pi}(\boldsymbol{\theta_0})}$ denotes the diagonal matrix with entries $\pi_j(\boldsymbol{\theta_0}),$ $j=1,...,L+1,$ and $\boldsymbol{\pi}(\boldsymbol{\theta}_0)^{\beta}$ denotes the vector with components $\pi_j(\boldsymbol{\theta}_0)^{\beta}.$
		
	\end{theorem} 
	
	The proof follows from Theorem 2 of Basu et al. (2018) and Result 3 of Balakrishnan et al. (2022).

	\section{Influence function of the restricted minimum density power divergence estimator}
	
	The robustness of an estimator is widely analyzed using the concept of Influence Function (IF), first introduced in Hampel et al. (1986). Intuitively, the IF describes the effect of an infinitesimal contamination of the model on the estimate. Therefore, IFs associated to locally	robust  estimators should be bounded. 
	The IF of the  MDPDE for the step-stress ALT model with non-destructive one-shot devices was established in Balakrishnan et al. (2022), and the boundedness of the function was discussing there, concluding that the IF of the  MDPDE is always bounded for positive values of the tuning parameter.
	Here we derive the IF of the restricted MDPDE, $\widetilde{\boldsymbol{\theta}}^\beta,$ defined in Section \ref{sec:modelformulation}. Observe that in this case the functional associated to the restricted estimator must also satisfied the subspace constraint.
	The statistical functional and influence function of the estimators under parametric restrictions have been rigorously studied in Ghosh (2015). Here, we study the IF of the restricted MDPDE when the subspace constraint has a linear form.
	We consider $F_{\theta}$ and the $G$ the assumed and real distribution functions with associated mass functions $\boldsymbol{\pi}(\boldsymbol{\theta})$ and $\boldsymbol{g},$ respectively.
	We define  $\widetilde{\boldsymbol{T}}_\beta$ the functional associated to the restricted MDPDE, $\widetilde{\boldsymbol{\theta}}^\beta,$ computed as the minimizer of the DPD given in (\ref{eq:DPDloss}) between the mass functions $\boldsymbol{\pi}(\boldsymbol{\theta})$ and $\boldsymbol{g}$ subject to the linear constraint $\boldsymbol{m}^T\boldsymbol{\theta} - d = 0.$
	
	For influence function analysis, one could derive the IF expression from the estimating equations of the restricted MDPDE in terms of Lagrange multipliers given in (\ref{eq:estimatingRMDPDE}). However, Ghosh (2015) proposed an alternative approach where the functional $\widetilde{\boldsymbol{T}}_\beta$ associated to the restricted MDPDE is calculated as a solution of the estimating equations of the (unconstrained) MDPDE over the subspace $\Theta_0.$ The existence of such solution is guaranteed by the Implicit Function Theorem.
	Hence, the IF of the restricted MDPDE at the contamination point $\boldsymbol{n}$ and the model distribution with true parameter value $\boldsymbol{\theta}_0$, $F_{\boldsymbol{\theta}_0},$  must simultaneously verify the expression of the IF of the MDPDE stated in Balakrishnan et al. (2022),
	\begin{equation*}
		\text{IF}\left(\boldsymbol{n}, \widetilde{\boldsymbol{T}}_\beta, F_{\boldsymbol{\theta}_0}\right) = \boldsymbol{J}_\beta^{-1}(\boldsymbol{\theta}_0) \boldsymbol{W}^T \boldsymbol{D}_{\boldsymbol{\pi}(\boldsymbol{\theta}_0)}^{\beta-1}\left(-\boldsymbol{\pi}(\boldsymbol{\theta}_0)+\Delta_{\boldsymbol{n}} \right)
	\end{equation*}
	and the subspace constraint
	$\boldsymbol{m}^T \widetilde{\boldsymbol{T}}_\beta - d = 0.$
	Differentiating on the previous subspace constraint, we have that
	$$\boldsymbol{m}^T \text{IF}\left(\boldsymbol{n}, \widetilde{\boldsymbol{T}}_\beta, F_{\boldsymbol{\theta}_0}\right) = 0$$
	and therefore, combining both equations, we get
	\begin{equation*}
		\begin{pmatrix}
			\boldsymbol{J}_\beta(\boldsymbol{\theta}_0)\\
			\boldsymbol{m}^T 
		\end{pmatrix} \text{IF}\left(\boldsymbol{n}, \widetilde{\boldsymbol{T}}_\beta, F_{\boldsymbol{\theta}_0} \right)
		= \begin{pmatrix}
			\boldsymbol{W}^T \boldsymbol{D}_{\boldsymbol{\pi}(\boldsymbol{\theta}_0)}^{\beta-1}\left(-\boldsymbol{\pi}(\boldsymbol{\theta}_0)+\Delta_{\boldsymbol{n}} \right) \\
			0
		\end{pmatrix}.
	\end{equation*}
	Now, multiplying both terms by $ \left(\boldsymbol{J}_\beta(\boldsymbol{\theta})^T, \boldsymbol{m} \right)$ and inverting in both sizes of the equation, the expression of the IF of the restricted MDPDE is given by
	\begin{equation}\label{eq:IFrestricted}
		\text{IF}\left(\boldsymbol{n}, \widetilde{\boldsymbol{T}}_\beta, F_{\boldsymbol{\theta}_0}\right)
		= \left(\boldsymbol{J}_\beta(\boldsymbol{\theta}_0)^T\boldsymbol{J}_\beta(\boldsymbol{\theta}_0) + \boldsymbol{m}\boldsymbol{m}^T \right)^{-1}
		\boldsymbol{J}_\beta(\boldsymbol{\theta}_0)^T
		\boldsymbol{W}^T \boldsymbol{D}_{\boldsymbol{\pi}(\boldsymbol{\theta}_0)}^{\beta-1}\left(-\boldsymbol{\pi}(\boldsymbol{\theta}_0)+\Delta_{\boldsymbol{n}} \right).
	\end{equation}
	Since the matrix $\left(\boldsymbol{J}_\beta(\boldsymbol{\theta}_0)^T\boldsymbol{J}_\beta(\boldsymbol{\theta}_0) + \boldsymbol{m}\boldsymbol{m}^T\right)^{-1}\boldsymbol{J}_\beta(\boldsymbol{\theta}_0)^T$ is typically assumed to be bounded, the robustness of the restricted MDPDE depends only on the boundedness of the IF of the (unrestricted) MDPDE. Therefore, restricted MDPDE are robust for all type of outliers when using positives values of $\beta$, whereas the restricted MLE (corresponding to $\beta=0$) lacks of robustness against stress level or inspection times contamination, i.e., bad leverage points.
	
	\section{Applications of the restricted MDPDE}
	
	An interesting application of the restricted MDPDE are robust testing procedures based on the DPD for testing linear null hypothesis of the form
	\begin{equation}\label{eq:null}
		\operatorname{H}_0: \boldsymbol{m}^T\boldsymbol{\theta} = d.
	\end{equation}
	with $\boldsymbol{m} \in \mathbb{R}^2$ and $d\in \mathbb{R}.$
	In this section, we develop two families of test statistics based on the DPD for testing (\ref{eq:null}) for one-shot devices under the step-stress ALT model, namely Rao-type test statistics and DPD-based tests statistics.
	These two families were studied for general statistical models in Basu et al. (2018).
 	Let consider $\widetilde{\boldsymbol{\theta}}^\beta$ the restricted MDPDE with restricted parameter space defined by the null hypothesis in (\ref{eq:null}),
	$$\Theta_0 = \{\boldsymbol{\theta} | \hspace{0.2cm} \boldsymbol{m}^T\boldsymbol{\theta} = d\}$$
	and recall   $\widehat{\boldsymbol{\theta}}^\beta$ denotes the MDPDE  for $\boldsymbol{\theta}$  computed in all parameter space.
	
	\subsection{Rao-type tests statistics}
	Let us consider the score of the DPD loss function for the step-stress ALT model
	\begin{equation} 
		\label{eq:score}
		\boldsymbol{U}_{\beta,N}(\boldsymbol{\theta}) = \boldsymbol{W}^T\boldsymbol{D}_{\boldsymbol{\pi}(\boldsymbol{\theta})}^{\beta-1}(\widehat{\boldsymbol{p}}-\boldsymbol{\pi}(\boldsymbol{\theta}))
	\end{equation}
	where matrices $\boldsymbol{W}$ and $\boldsymbol{D}_{\boldsymbol{\pi}(\boldsymbol{\theta})}$ are defined in Section \ref{sec:modelformulation}.	That is, the MDPDE verifies the estimating equations given by
	$$	\boldsymbol{U}_{\beta,N}(\widehat{\boldsymbol{\theta}}^\beta) = \boldsymbol{0}$$
	
	We define Rao-type test statistics for testing linear null hypothesis (\ref{eq:null}) as
	\begin{definition}
		The Rao-type statistics, based on the restricted to the linear null hypothesis (\ref{eq:null}) MDPDE, $\widetilde{\boldsymbol{\theta}}^\beta,$ for testing (\ref{eq:null}) is given by
		\begin{equation}\label{eq:raotest}
			\boldsymbol{R}_{\beta, N}(\widetilde{\boldsymbol{\theta}}^\beta) = N\boldsymbol{U}_{\beta,N}(\widetilde{\boldsymbol{\theta}}^\beta)^T \boldsymbol{Q}_{\beta}(\widetilde{\boldsymbol{\theta}}^\beta)\left[ \boldsymbol{Q}_{\beta}(\widetilde{\boldsymbol{\theta}}^\beta)^T \boldsymbol{K}_{\beta}(\widetilde{\boldsymbol{\theta}}^\beta)
			\boldsymbol{Q}_{\beta}(\widetilde{\boldsymbol{\theta}}^\beta)\right]^{-1}
			\boldsymbol{Q}_{\beta}(\widetilde{\boldsymbol{\theta}}^\beta)^T \boldsymbol{U}_{\beta,N}(\widetilde{\boldsymbol{\theta}}^\beta),
		\end{equation}
		where matrices $\boldsymbol{K}_{\beta}(\boldsymbol{\theta})$, $\boldsymbol{Q}_{\beta}(\boldsymbol{\theta})$
		and $\boldsymbol{U}_{\beta,N}(\boldsymbol{\theta}),$  is defined in (\ref{eq:JK}), (\ref{eq:SigmaPQmatrices}) and (\ref{eq:score}), respectively.
	\end{definition}
	Here, the matrix $\boldsymbol{Q}_{\beta}(\boldsymbol{\theta})$ depends on the null hypothesis trough $\boldsymbol{m},$ and the term $d$ is only used to obtain the restricted MDPDE. 
	Moreover, if $\boldsymbol{m} = (0,1)$ (simple null hypothesis) then the restricted estimate of $\theta_1$ must be necessarily $d.$ 
	
	Before presenting the asymptotic distribution of the Rao-type test statistics, $\boldsymbol{R}_{\beta, N}(\widetilde{\boldsymbol{\theta}}^\beta),$ we shall establish the asymptotic distribution of the score  $\boldsymbol{U}_{\beta,N}(\widetilde{\boldsymbol{\theta}}^\beta).$ 
	\begin{theorem}\label{thm:asympscore}
		The asymptotic distribution of the score $\boldsymbol{U}_{\beta,N}(\widetilde{\boldsymbol{\theta}}^\beta)$ for the step-stress ALT model under exponential lifetimes, is given by
		$$\sqrt{N}\boldsymbol{U}_{\beta,N}(\widetilde{\boldsymbol{\theta}}^\beta) \xrightarrow[N\rightarrow \infty]{L} \mathcal{N}\left(\boldsymbol{0}, \boldsymbol{K}_\beta(\boldsymbol{\theta})\right)$$
		where the variance-covariance matrix $\boldsymbol{K}_\beta(\boldsymbol{\theta})$ is defined in (\ref{eq:JK}).
	\end{theorem}
	\begin{proof}
		It is well known that $$\sqrt{N}(\widehat{\boldsymbol{p}}-\boldsymbol{\pi}(\boldsymbol{\theta})) \xrightarrow[N\rightarrow \infty]{L} \mathcal{N}\left(\boldsymbol{0}, \boldsymbol{D}_{\boldsymbol{\pi}(\boldsymbol{\theta})} - \boldsymbol{\pi}(\boldsymbol{\theta}) \boldsymbol{\pi}(\boldsymbol{\theta})^T\right).$$
		since $\widehat{\boldsymbol{p}}$ is the MLE of the multinomial model, and $\boldsymbol{D}_{\boldsymbol{\pi}(\boldsymbol{\theta})} - \boldsymbol{\pi}(\boldsymbol{\theta}) \boldsymbol{\pi}(\boldsymbol{\theta})^T$ is the inverse of the Fisher information matrix of that model.
		Therefore, the score $\sqrt{N}\boldsymbol{U}_{\beta,N}(\boldsymbol{\theta})$ is asymptotically normal with mean vector 
		$$\mathbb{E}\left[ \boldsymbol{U}_{\beta,N}(\boldsymbol{\theta}) \right] =\mathbb{E}\left[ \boldsymbol{W}^T\boldsymbol{D}_{\boldsymbol{\pi}(\boldsymbol{\theta})}^{\beta-1}(\widehat{\boldsymbol{p}}-\boldsymbol{\pi}(\boldsymbol{\theta})) \right] = \boldsymbol{0}$$
		and variance-covariance matrix
		\begin{align*}
			\text{Cov}\left[\boldsymbol{U}_{\beta,N}(\boldsymbol{\theta})\right] &=  \boldsymbol{W}^T\boldsymbol{D}_{\boldsymbol{\pi}(\boldsymbol{\theta})}^{\beta-1}\text{Cov}\left[\widehat{\boldsymbol{p}}-\boldsymbol{\pi}(\boldsymbol{\theta})\right]\boldsymbol{D}_{\boldsymbol{\pi}(\boldsymbol{\theta})}^{\beta-1} \boldsymbol{W}\\
			&= \boldsymbol{W}^T\boldsymbol{D}_{\boldsymbol{\pi}(\boldsymbol{\theta})}^{\beta-1}\left[\boldsymbol{D}_{\boldsymbol{\pi}(\boldsymbol{\theta})} - \boldsymbol{\pi}(\boldsymbol{\theta}) \boldsymbol{\pi}(\boldsymbol{\theta})^T\right]\boldsymbol{D}_{\boldsymbol{\pi}(\boldsymbol{\theta})}^{\beta-1} \boldsymbol{W}\\
			&=  \boldsymbol{W}^T\left[\boldsymbol{D}_{\boldsymbol{\pi}(\boldsymbol{\theta})}^{2\beta-1} - \boldsymbol{D}_{\boldsymbol{\pi}(\boldsymbol{\theta})}^{\beta-1}\boldsymbol{\pi}(\boldsymbol{\theta}) \boldsymbol{\pi}(\boldsymbol{\theta})^T\boldsymbol{D}_{\boldsymbol{\pi}(\boldsymbol{\theta})}^{2\beta-1}\right] \boldsymbol{W}\\
			&=  \boldsymbol{W}^T\left[\boldsymbol{D}_{\boldsymbol{\pi}(\boldsymbol{\theta})}^{2\beta-1} - \boldsymbol{\pi}(\boldsymbol{\theta})^\beta \boldsymbol{\pi}(\boldsymbol{\theta})^{\beta T}\right] \boldsymbol{W}\\
			&= \boldsymbol{K}_\beta(\boldsymbol{\theta}).
		\end{align*}
		
	\end{proof}
	
	Now, the following results states the asymptotic distribution of the Rao-type test statistics
	\begin{theorem}\label{thm:asymprao}
		The asymptotic distribution of the Rao-type test statistics defined in (\ref{eq:raotest}) under the linear null hypothesis (\ref{eq:null}) is a chi-square with 1 degree of freedom.
	\end{theorem}
	\begin{proof}
		The MDPDE restricted to the null hypothesis (\ref{eq:null}), $\widetilde{\boldsymbol{\theta}}^\beta,$ 
		is defined as the minimum of the DPD loss restricted to the condition $\boldsymbol{m}^T\boldsymbol{\theta} =d$. Then, it satisfies the restricted equations
		\begin{equation}\label{eq:restrictedestimating}
			\widetilde{\boldsymbol{U}}_{\beta,N}(\widetilde{\boldsymbol{\theta}}^\beta) + \boldsymbol{m}\widetilde{\boldsymbol{\lambda}} = \boldsymbol{0},
		\end{equation}
		for some vector $\widetilde{\boldsymbol{\lambda}}$ of Lagrangian multipliers. 
		Then, we can write
		$\boldsymbol{U}_{\beta,N}(\widetilde{\boldsymbol{\theta}}^\beta) = -\boldsymbol{m}\widetilde{\boldsymbol{\lambda}}
		$
		and consequently
		$$
		\boldsymbol{U}_{\beta,N}(\widetilde{\boldsymbol{\theta}}^\beta)^T \boldsymbol{Q}_{\beta}(\widetilde{\boldsymbol{\theta}}^\beta) = -\widetilde{\boldsymbol{\lambda}}\boldsymbol{m}^T\boldsymbol{Q}_{\beta}(\widetilde{\boldsymbol{\theta}}^\beta) = - \boldsymbol{\widetilde{\lambda}}
		$$
		where $\boldsymbol{Q}_{\beta}(\widetilde{\boldsymbol{\theta}}^\beta)$ is defined in (\ref{eq:SigmaPQmatrices}). Further, the Rao-type test statistics defined in (\ref{eq:raotest}) can be computed in terms of the vector of Lagrange multipliers as follows,
		\begin{equation*}
			\boldsymbol{R}_{\beta, N}(\boldsymbol{\theta}) = N\boldsymbol{\widetilde{\lambda}}^T\left[ \boldsymbol{Q}_{\beta}(\widetilde{\boldsymbol{\theta}}^\beta)^T \boldsymbol{K}_{\beta}(\widetilde{\boldsymbol{\theta}}^\beta)
			\boldsymbol{Q}_{\beta}(\widetilde{\boldsymbol{\theta}}^\beta)\right]^{-1}
			\boldsymbol{\widetilde{\lambda}}.
		\end{equation*}
		In order to obtain the asymptotic distribution of the Rao-type test statistics, we first derive the asymptotic distribution of the vector of Lagrangian multipliers $\widetilde{\boldsymbol{\lambda}}.$
		We  consider the second order Taylor expansion series of the score function $\boldsymbol{U}_{\beta,N}(\boldsymbol{\theta})$ around the true parameter value $\boldsymbol{\theta}_0,$ 
		$$\boldsymbol{U}_{\beta,N}(\widetilde{\boldsymbol{\theta}}^\beta) = \boldsymbol{U}_{\beta,N}(\boldsymbol{\theta}_0) + \frac{\partial \boldsymbol{U}_{\beta,N}(\boldsymbol{\theta})}{\partial \boldsymbol{\theta}}\bigg|_{\boldsymbol{\theta} = \boldsymbol{\theta}_0}\left(\widetilde{\boldsymbol{\theta}}^\beta- \boldsymbol{\theta}_0\right) + o\left(||\widetilde{\boldsymbol{\theta}}^\beta- \boldsymbol{\theta}_0||^2\boldsymbol{1}_2\right).$$
		On the other hand, since $\widehat{p} \xrightarrow[N\rightarrow \infty]{P} \boldsymbol{\pi}(\boldsymbol{\theta}_0)$, it is not difficult to show that 
		$$\frac{\partial \boldsymbol{U}_{\beta,N}(\boldsymbol{\theta})}{\partial \boldsymbol{\theta}}\bigg|_{\boldsymbol{\theta} = \boldsymbol{\theta}_0} \xrightarrow[N\rightarrow \infty]{P} - \boldsymbol{J}_\beta\left(\boldsymbol{\theta}_0\right),$$
		where the matrix $\boldsymbol{J}_\beta\left(\boldsymbol{\theta}\right)$ is defined in (\ref{eq:JK}). 
		Therefore, we can approximate the score function at the restricted MDPDE by
		$$\boldsymbol{U}_{\beta,N}(\widetilde{\boldsymbol{\theta}}^\beta) = \boldsymbol{U}_{\beta,N}(\boldsymbol{\theta}_0) - \boldsymbol{J}_\beta\left(\boldsymbol{\theta}_0\right) \left(\widetilde{\boldsymbol{\theta}}^\beta- \boldsymbol{\theta}_0\right) + o\left(||\widetilde{\boldsymbol{\theta}}^\beta- \boldsymbol{\theta}_0||^2\boldsymbol{1}_2\right)+o\left( \boldsymbol{1}_2\right).$$
		Based on (\ref{eq:restrictedestimating}), we have that
		$$
		\boldsymbol{U}_{\beta,N}(\boldsymbol{\theta}_0) - \boldsymbol{J}_\beta\left(\boldsymbol{\theta}_0\right) \left(\widetilde{\boldsymbol{\theta}}^\beta- \boldsymbol{\theta}_0\right)  + \boldsymbol{m}\widetilde{\boldsymbol{\lambda}}= o\left( \boldsymbol{1}_2\right)
		$$
		so under the null hypothesis we can write,
		$$\boldsymbol{m}^T\widetilde{\boldsymbol{\theta}}^\beta - d = \boldsymbol{m}^T\left(\widetilde{\boldsymbol{\theta}}^\beta - \boldsymbol{\theta}_0\right)  = 0. $$
		Joining both equations,  we get
		\begin{equation*}
			\begin{pmatrix}
				-\boldsymbol{J}_\beta(\boldsymbol{\theta}_0) & \boldsymbol{m}\\
				\boldsymbol{m}^T & \boldsymbol{0}
			\end{pmatrix} 
			\begin{pmatrix}
				\widetilde{\boldsymbol{\theta}}^\beta - \boldsymbol{\theta}_0\\
				\widetilde{\boldsymbol{\lambda}}
			\end{pmatrix} =
			\begin{pmatrix}
				-\boldsymbol{U}_{\beta,N}(\boldsymbol{\theta}_0)\\
				0
			\end{pmatrix} + 	\begin{pmatrix}
				o(\boldsymbol{1}_2)\\
				0
			\end{pmatrix} 
		\end{equation*}
		and solving the previous equation, we have that
		\begin{equation*}
			\begin{pmatrix}
				\widetilde{\boldsymbol{\theta}}^\beta - \boldsymbol{\theta}_0\\
				\widetilde{\boldsymbol{\lambda}}
			\end{pmatrix} =
			\begin{pmatrix}
				-\boldsymbol{J}_\beta(\boldsymbol{\theta}_0) & \boldsymbol{m}\\
				\boldsymbol{m}^T & \boldsymbol{0}
			\end{pmatrix}^{-1}
			\begin{pmatrix}
				-\boldsymbol{U}_{\beta,N}(\boldsymbol{\theta}_0)\\
				\boldsymbol{0}
			\end{pmatrix}  + 	\begin{pmatrix}
				o(\boldsymbol{1}_2)\\
				0
			\end{pmatrix} .
		\end{equation*}
		Now, computing the inverse matrix
		\begin{equation*}
			\begin{pmatrix}
				-\boldsymbol{J}_\beta(\boldsymbol{\theta}_0) & \boldsymbol{m}\\
				\boldsymbol{m}^T & \boldsymbol{0}
			\end{pmatrix}^{-1} = 
			\begin{pmatrix}
				\boldsymbol{P}_\beta(\boldsymbol{\theta}_0) & \boldsymbol{Q}_\beta(\boldsymbol{\theta}_0)\\
				\boldsymbol{Q}_\beta(\boldsymbol{\theta}_0)^T & \left(\boldsymbol{m}^T\boldsymbol{J}_\beta(\boldsymbol{\theta}_0)^{-1}\boldsymbol{m}\right)^{-1}
			\end{pmatrix}
		\end{equation*} 
		with $\boldsymbol{P}_\beta(\boldsymbol{\theta}_0)$ and $\boldsymbol{Q}_\beta(\boldsymbol{\theta}_0)$ defined in (\ref{eq:SigmaPQmatrices}). But from Theorem \ref{thm:asympscore}, 
		\begin{equation*}
			\begin{pmatrix}
				\sqrt{N} \boldsymbol{U}_{\beta,N}(\boldsymbol{\theta}_0)\\
				\boldsymbol{0}
			\end{pmatrix} \xrightarrow[N\rightarrow \infty]{L} \mathcal{N}\left(\boldsymbol{0}_{3}, \begin{pmatrix}
				\boldsymbol{K}_\beta(\boldsymbol{\theta}_0) & \boldsymbol{0}\\
				\boldsymbol{0}^T & 0
			\end{pmatrix}\right)
		\end{equation*}
		and hence, 
		\begin{equation*}
			\begin{pmatrix}
				\sqrt{N}(\widetilde{\boldsymbol{\theta}}^\beta - \boldsymbol{\theta}_0)\\
				\sqrt{N}\widetilde{\boldsymbol{\lambda}}
			\end{pmatrix} \xrightarrow[N\rightarrow \infty]{L} \mathcal{N}\left(\boldsymbol{0}_{3},
			\boldsymbol{V}_\beta(\boldsymbol{\theta}_0) \right)
		\end{equation*}
		with 
		\begin{equation*}
			\boldsymbol{V}_\beta(\boldsymbol{\theta}_0) = \begin{pmatrix}
				\boldsymbol{P}_\beta(\boldsymbol{\theta}_0) & \boldsymbol{Q}_\beta(\boldsymbol{\theta}_0)\\
				\boldsymbol{Q}_\beta(\boldsymbol{\theta}_0)^T & \left(\boldsymbol{m}^T\boldsymbol{J}_\beta(\boldsymbol{\theta}_0)^{-1}\boldsymbol{m}\right)^{-1}
			\end{pmatrix}
			\begin{pmatrix}
				\boldsymbol{K}_\beta(\boldsymbol{\theta}_0) & \boldsymbol{0}\\
				\boldsymbol{0}^T & 0
			\end{pmatrix}
			\begin{pmatrix}
				\boldsymbol{P}_\beta(\boldsymbol{\theta}_0) & \boldsymbol{Q}_\beta(\boldsymbol{\theta}_0)\\
				\boldsymbol{Q}_\beta(\boldsymbol{\theta}_0)^T & \left(\boldsymbol{m}^T\boldsymbol{J}_\beta(\boldsymbol{\theta}_0)^{-1}\boldsymbol{m}\right)^{-1}
			\end{pmatrix}.
		\end{equation*}
		Thus, the asymptotic distribution of the vector of Lagrangian multipliers is given by
		\begin{equation*}\label{eq:asymplambda}
			\sqrt{N} \widetilde{\boldsymbol{\lambda}}
			\xrightarrow[N\rightarrow \infty]{L} \mathcal{N}\left(0,
			\boldsymbol{Q}_\beta(\boldsymbol{\theta}_0)^T \boldsymbol{K}_\beta(\boldsymbol{\theta}_0)\boldsymbol{Q}_\beta(\boldsymbol{\theta}_0) \right).
		\end{equation*}
		Using the previous convergence and the consistency of the restricted MDPDE, it follows that the asymptotic distribution of the Rao-type test statistics, $\boldsymbol{R}_{\beta, N}(\widetilde{\boldsymbol{\theta}}^\beta),$
		is a chi-square distribution with 1 degree of freedom.
		
	\end{proof}
	
	Based on Theorem \ref{thm:asymprao}, for any $\beta \geq 0$ and $\boldsymbol{m} \in \mathbb{R}^2,$ the critical region with significance level $\alpha$ for the hypothesis test with null hypothesis (\ref{eq:null}) is given by
	\begin{equation}\label{eq:criticalregionRaotest}
		\mathcal{R}_{\alpha} = \{(n_1,...,n_{L+1})  \text{ s.t. } 	\boldsymbol{R}_{\beta, N}(\boldsymbol{\theta}) > \chi^2_{1,\alpha}\}
	\end{equation}
	where $\chi^2_{1,\alpha}$ denotes the lower $\alpha$-quantile of a chi-square with 1 degree of freedom.
	
	\begin{remark}
		We could extent the scope to more general linear null hypothesis,
		$$\operatorname{H}_0: \boldsymbol{M}^T\boldsymbol{\theta} = \boldsymbol{d},$$
		with $\boldsymbol{M}$ a $2\times2$ matrix of range $2$ and $\boldsymbol{d}\in \mathbb{R}^2.$ However, in this case the restricted MDPDE is explicitly determinate by the constraint $\boldsymbol{M}^T\boldsymbol{\theta} = \boldsymbol{d}.$ 
		Consequently, the Rao-type test statistics is completely defined by the null hypothesis and it is independent of the value of $\beta$.
		In particular, if we define $\boldsymbol{M}$ to be the identity matrix, we get the simple null hypothesis 
		$$\operatorname{H}_0: (\theta_0, \theta_1) = (d_0,d_1),$$
		and the associated Rao-type test is given by
		\begin{equation*}
			\boldsymbol{R}_{\beta, N}^\ast(\boldsymbol{d}) = N\boldsymbol{U}_{\beta,N}(\boldsymbol{d})^T  \boldsymbol{K}_{\beta}(\boldsymbol{d})^{-1}
			\boldsymbol{U}_{\beta,N}(\boldsymbol{d}),
		\end{equation*}
		since $\boldsymbol{Q}_{\beta}(\boldsymbol{d})$ is the identity matrix.
		Following similar steps than in Theorem \ref{thm:asymprao}, it is not difficult to establish that
		$$ \boldsymbol{R}_{\beta, N}^\ast(\boldsymbol{d}) \xrightarrow[N \rightarrow \infty]{L} \chi^2_2.$$
		Thus, for any $\beta \geq 0,$  the critical region with significance level $\alpha$ for the  hypothesis test with simple null hypothesis is given by
		\begin{equation*}
			\mathcal{R}_{\alpha} = \{(n_1,...,n_{L+1})  \text{ s.t. } 	\boldsymbol{R}_{\beta, N}^\ast(\boldsymbol{d}) > \chi_{2,\alpha}\},
		\end{equation*}
		where $\chi_{2,\alpha}$ denotes the lower $\alpha$-quantile of a chi-square with 2 degree of freedom.
	\end{remark}

	\subsection{DPD-based statistics}
	
	We develop a class of test statistics based on the DPD between the model evaluated under the null hypothesis and under the whole parameter space, respectively.
	Let $\boldsymbol{\pi}(\widehat{\boldsymbol{\theta}}^\beta)$ and $\boldsymbol{\pi}(\widetilde{\boldsymbol{\theta}}^\beta)$ denote the probability vector of the multinomial model estimated in the whole parameter space and the restricted parameter space, respectively. The DPD for $\tau > 0 $ between these two mass function is given by
	 	\begin{equation}\label{eq:DPDtest}
	 	d_\tau(\widehat{\boldsymbol{\theta}}^\beta, \widetilde{\boldsymbol{\theta}}^\beta) = d_\tau(\boldsymbol{\pi}(\widehat{\boldsymbol{\theta}}^\beta), \boldsymbol{\pi}(\widetilde{\boldsymbol{\theta}}^\beta))
 		= \sum_{i=1}^{L+1}\left(\pi_j(\widehat{\boldsymbol{\theta}}^\beta)^{1+\tau} + \left(1+\frac{1}{\tau}\right)\pi_j(\widehat{\boldsymbol{\theta}}^\beta) \pi_j(\widetilde{\boldsymbol{\theta}}^\beta)^\tau + \frac{1}{\tau}\pi_j(\widetilde{\boldsymbol{\theta}}^\beta)^{\tau+1}\right).
	 	\end{equation}
	 
	 Note that here we should deal with two different tuning parameters controlling the trade-off between robustness and efficiency in the estimation ($\beta$) and in the DPD between the estimated mass functions ($\tau$). Many authors choose $\tau = \beta$ for the seek of simplicity.
	 If the true parameter $\boldsymbol{\theta}_0$ verifies the null hypothesis, then the distance in (\ref{eq:DPDtest}) must near zero. Then, the reject region for the test with null hypothesis (\ref{eq:null}) is given by
	 $$ RC= \{ (n_1,...,n_{L+1}) s.t.\hspace{0.2cm} d_\tau(\widehat{\boldsymbol{\theta}}^\beta, \widetilde{\boldsymbol{\theta}}^\beta)  > k\}$$
	with $k$ a constant computed such that the test has size $\alpha.$
	The exact distribution of the test statistic $d_\tau(\widehat{\boldsymbol{\theta}}^\beta, \widetilde{\boldsymbol{\theta}}^\beta)$ 
	is not easy to get, but using similar arguments as in Jaenada et al. (2022) is it not difficult to establish that
	$$T_\tau(\widehat{\boldsymbol{\theta}}^\beta, \widetilde{\boldsymbol{\theta}}^\beta) = 2n d_\tau(\widehat{\boldsymbol{\theta}}^\beta, \widetilde{\boldsymbol{\theta}}^\beta)$$ is asymptotically distributed, under the null hypothesis, as a linear combination of chi-squared random variables with a degree of freedom.
	Furthermore, asymptotic results under some particular kind of alternative hypothesis, such as contiguous alternative hypothesis, may be also of interest.

\section{Simulation study}\label{sec:simstudy}

We empirically analyze the performance of the restricted MDPDEs for one-shot devices data under the step-stress model with exponential lifetime  distributions, $\widetilde{\boldsymbol{\theta}}^\beta.$ 
Further, we evaluate their robustness properties under different scenarios of  contamination.

For the multinomial model, we should consider ``outlying cells'' rather than  ``outlying devices''. Then, we introduce contamination by  increasing (or decreasing) the probability of failure in (\ref{eq:th.prob}) for (at least) one interval (i.e., one cell).
The probability of failure is switched in such contaminated cells as
\begin{equation} \label{eq:contaminatedprob}
	\tilde{\pi}_j(\boldsymbol{\theta}) =  G_{\boldsymbol{\theta}}(\textit{IT}_j) - G_{\boldsymbol{\tilde{\theta}}}(\textit{IT}_{j-1})
\end{equation}
for some $j=2,...,L$, where $\boldsymbol{\tilde{\theta}} = (\widetilde{\theta}_0,\widetilde{\theta}_1)$ is a contaminated parameter with $\tilde{\theta}_0 \leq \theta_0$  and $\tilde{\theta}_1 \leq \theta_1.$
The resulting probability vector must normalized after introducing contamination.

Let us consider a 2-step stress ALT experiment with $L=11$ inspection times and a total of $N = 180$ one-shot devices.
The devices are tested at two stress levels,  $x_1=35$ and $x_2=45$. The first stress level is maintained from the beginning of the experiment until $\tau_1 = 25,$ and the experiment ends at $\tau_2=70.$
 During the experiment, devices inspection is performed at a grid of inspection times, containing the times of stress change, $\text{IT}=(10,15,20,25,30,35,40,45,50,60,70).$
 At each inspection time, all surviving units under test are examined.
 We set the true value of the parameter $\boldsymbol{\theta}_0=(0.003,0.03),$ and then generate data from the corresponding
 multinomial model described in Section \ref{sec:modelformulation}.
  Additionally, we contaminate the described model by increasing the probability of failure in the third interval following (\ref{eq:contaminatedprob}), with $\widetilde{\boldsymbol{\theta}} = (\varepsilon\theta_0, \theta_1)$ for the first scenario of contamination and $\widetilde{\boldsymbol{\theta}} = (\theta_0, \varepsilon\theta_1)$ for the second scenario of contamination.
Note that in both scenarios, the mean lifetime is decreased for the outlying cell.
Further, we consider the linearly restricted parameter space of the form
\begin{equation*}
	\Theta_0 = \{\theta \in \Theta | \hspace{0.2cm} \theta_1 = d \}
\end{equation*} 
with $d \in \mathbb{R}.$ 
We evaluate the performance of the restricted estimators when the true parameter verifies the subspace constraints as well as when $\boldsymbol{\theta}_0$ does not belong to the restricted subspace $\Theta_0.$
For the first scenario ($\boldsymbol{\theta}_0 \in \Theta_0$) we set $d=0.03$ and for the second one ($\boldsymbol{\theta}_0 \notin \Theta_0$) we set $d=0.027.$ Note that, in the second case, the true parameter is close to the restricted space.

Figure \ref{fig:restricted} shows the mean squared error (MSE) of the restricted MDPDE when the true parameter value satisfies the subspace restrictions over $R=10.000$ repetitions, for the two scenarios of contamination and different values of $\beta.$ Larger values of the tuning parameter produces more robust estimates, although less efficient. From this empirical results, moderately large values of the tuning parameter $\beta$ over 0.4 could produce the best trade-off between efficiency and robustness.
Similarly, Figure  \ref{fig:restrictedd} shows the MSE of the restricted MDPDE for similar scenarios, when the true parameter value does not belong to the restricted parameter space. Similar conclusions are drawn, but the MSE on the estimation increases considerably in this context.

\begin{figure}[htb]
	\begin{subfigure}{0.5\textwidth}
		\includegraphics[scale=0.4]{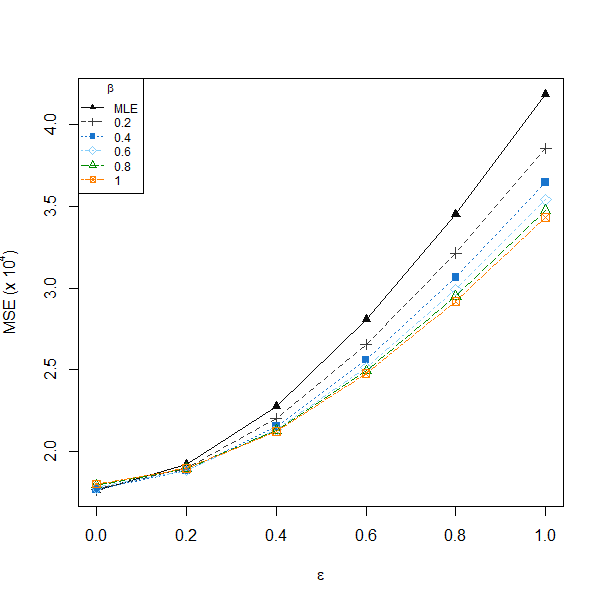} 
		\subcaption{$\theta_0$-contaminated cell}
	\end{subfigure}
	\begin{subfigure}{0.5\textwidth}
		\includegraphics[scale=0.4]{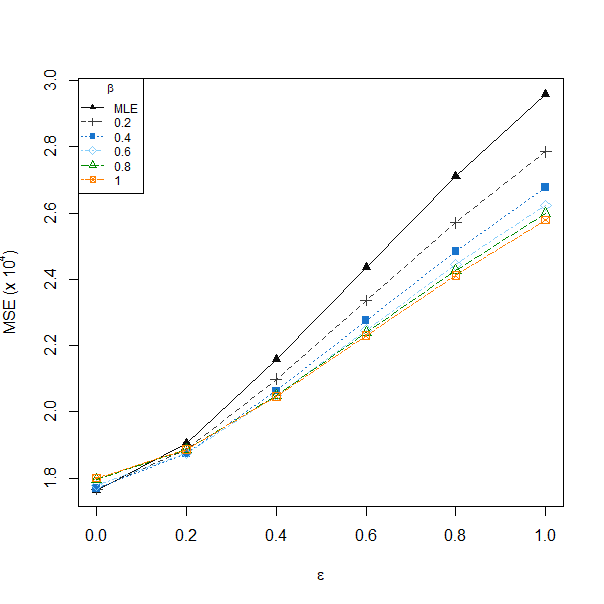}
		\subcaption{$\theta_1$-contaminated cell}
	\end{subfigure}
	\caption{ MSE of the restricted MDPDE when the true parameter belongs to the restricted space over $R=10.000$ repetitions}
\label{fig:restricted}
\end{figure}

\begin{figure}[htb]
	\begin{subfigure}{0.5\textwidth}
		\includegraphics[scale=0.4]{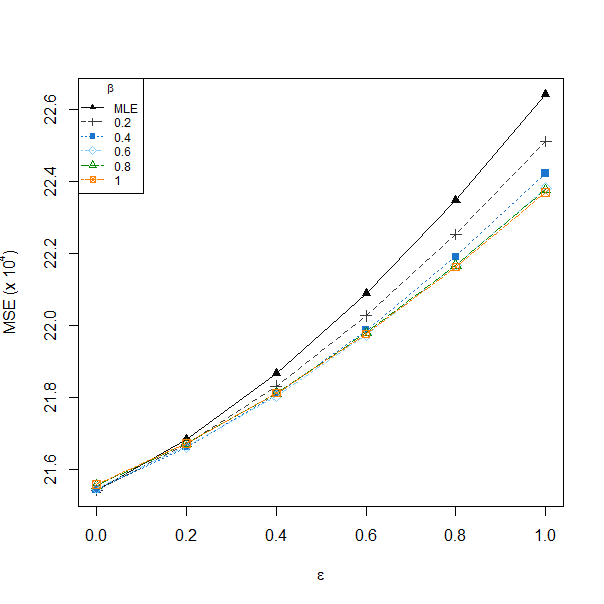} 
		\subcaption{$\theta_0$-contaminated cell}
	\end{subfigure}
	\begin{subfigure}{0.5\textwidth}
		\includegraphics[scale=0.4]{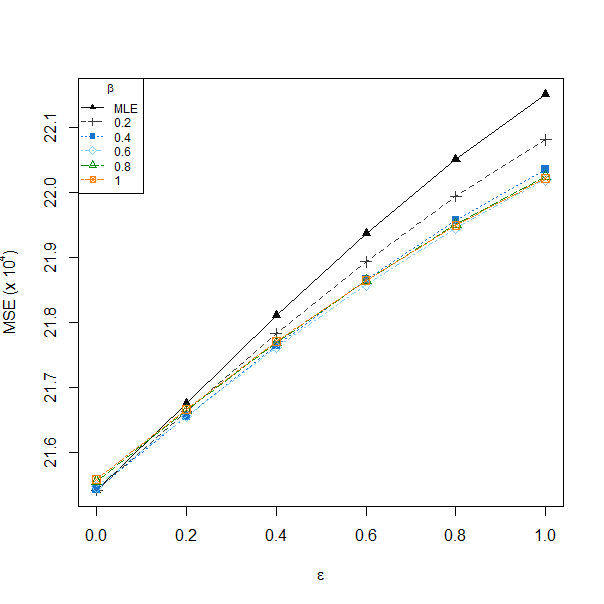}
		\subcaption{$\theta_1$-contaminated cell}
	\end{subfigure}
	\caption{ MSE of the restricted MDPDE when the true parameter does not belong to the restricted space over $R=10.000$ repetitions}
	\label{fig:restrictedd}
\end{figure}

Finally, we examine the performance of the Rao-type test statistics against the sample size $N$ in three contamination scenarios (Figure \ref{fig:raolevelandpowerN}), in the absence of contamination (top), $\theta_0$-contaminated third cell with a $40\%$ reduction of the parameter (middle) and $\theta_1$-contaminated third cell with a $40\%$ reduction of the parameter (bottom). Again, we found a similar performance of the Rao-type test statistics based on the restricted MDPDE with different values of the tuning parameter $\beta,$ whereas positives values of $\beta$ produce more robust statistics when introducing contamination in either of the model parameters, $\theta_0$ and $\theta_1$. Moreover, the lack of robustness of the Rao-type test statistics based on the restricted MLE is highlighted with large sample sizes, in both empirical level and power.

\begin{figure}[htb]
	\begin{subfigure}{0.5\textwidth}
		\includegraphics[height=7cm, width=8cm]{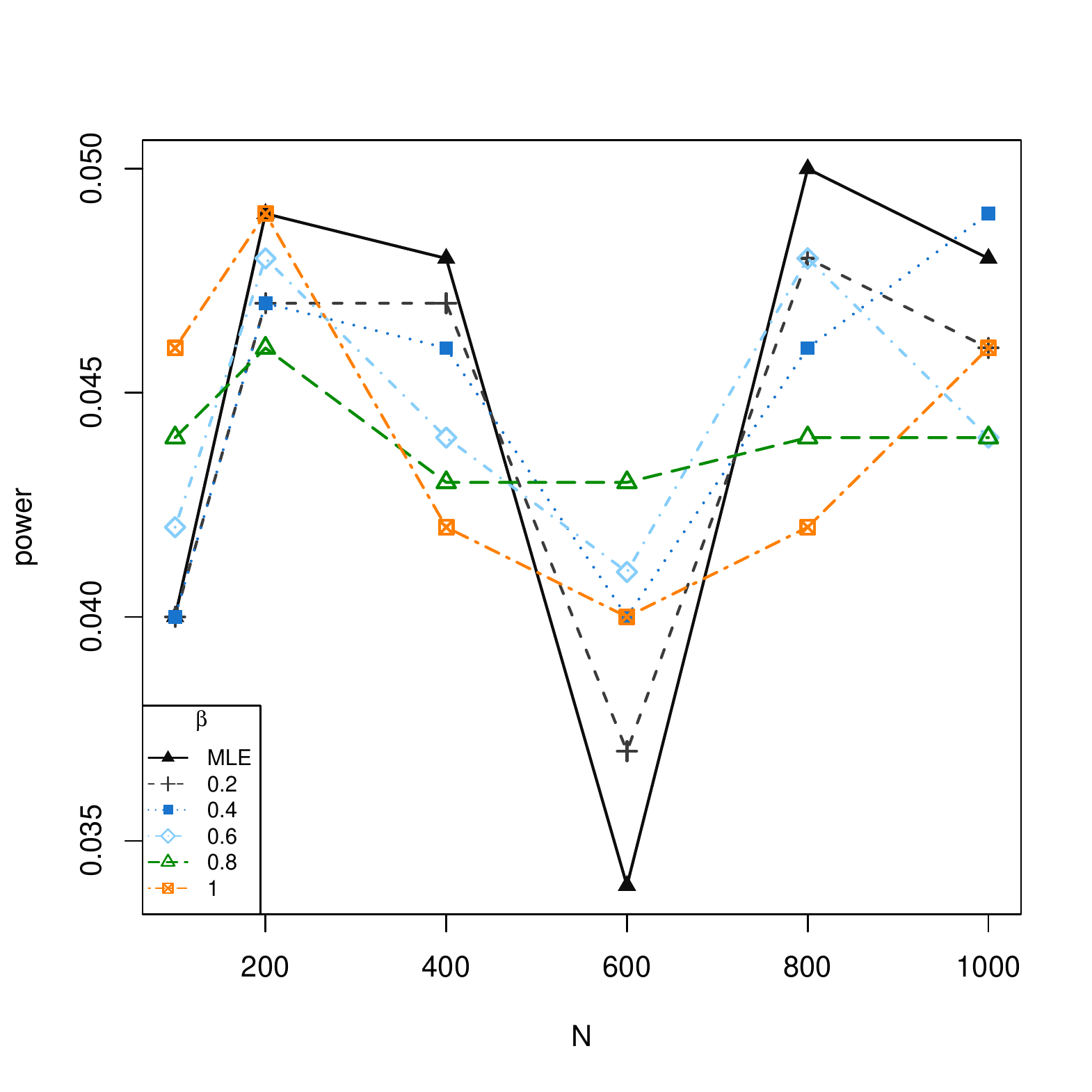}
		\subcaption{Absence of contamination}
	\end{subfigure}
	\begin{subfigure}{0.5\textwidth}
		\includegraphics[height=7cm, width=8cm]{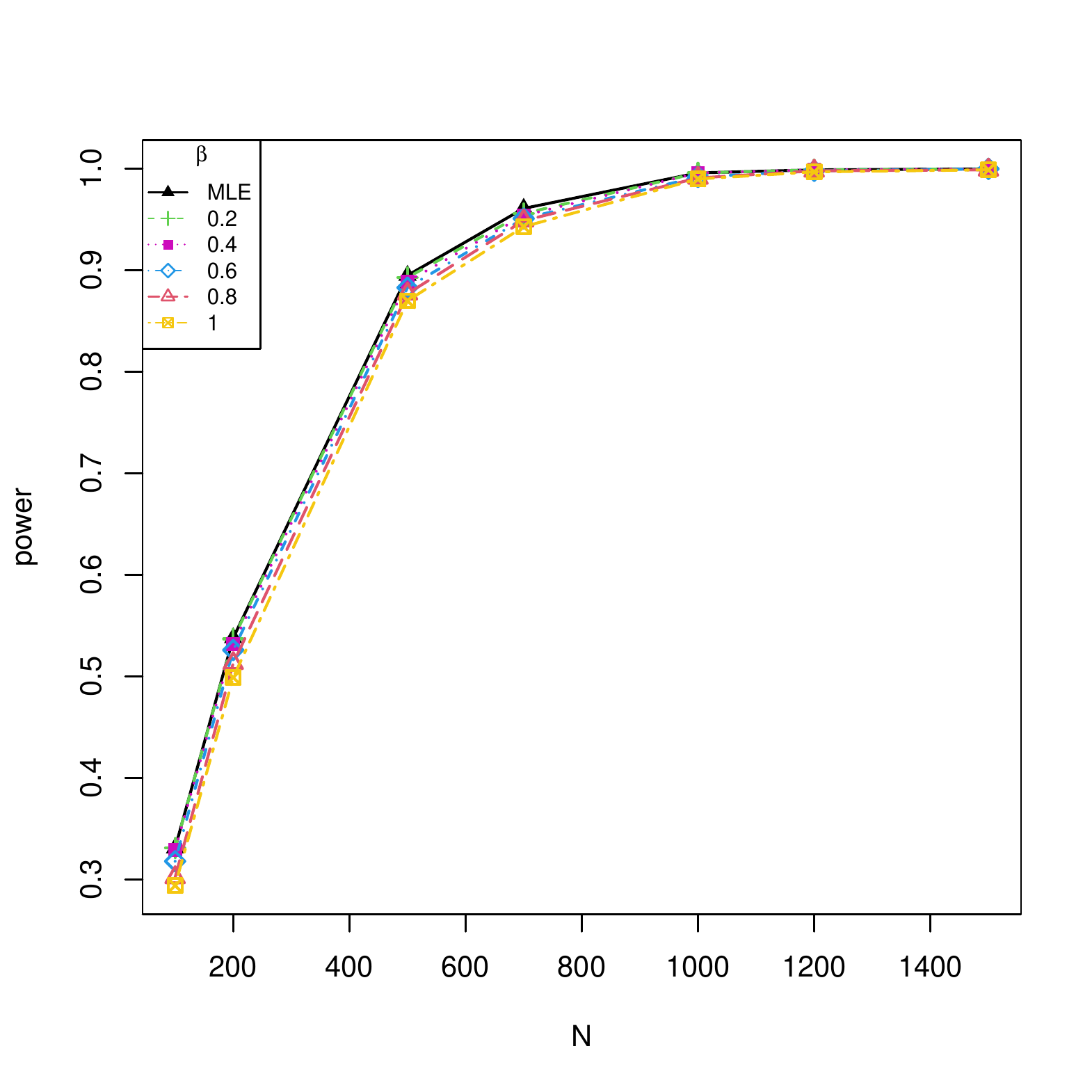}
		\subcaption{Absence of contamination}
	\end{subfigure}
	\begin{subfigure}{0.5\textwidth}
		\includegraphics[height=7cm, width=8cm]{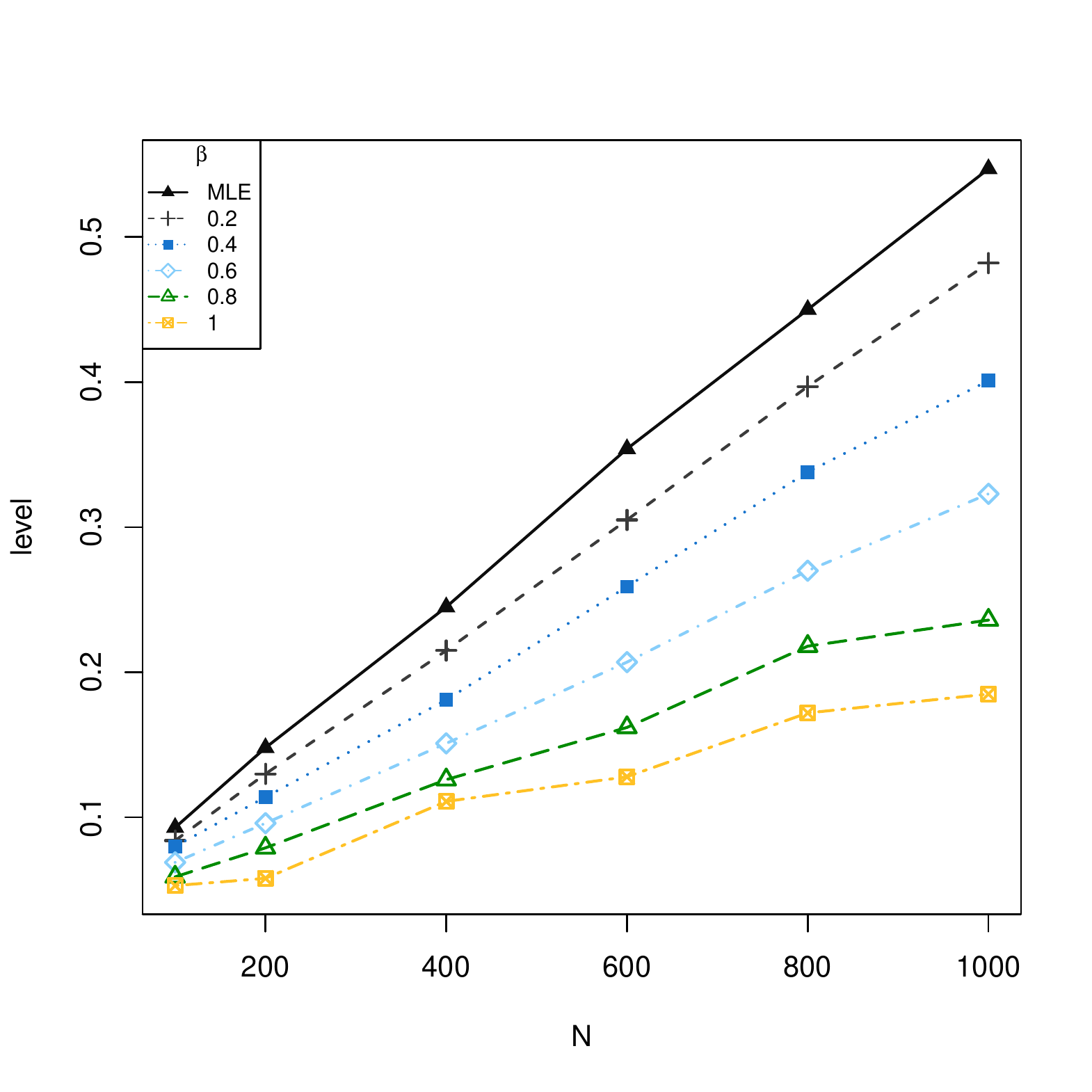}
		\subcaption{$\theta_0$-contaminated cell}
	\end{subfigure}
	\begin{subfigure}{0.5\textwidth}
		\includegraphics[height=7cm, width=8cm]{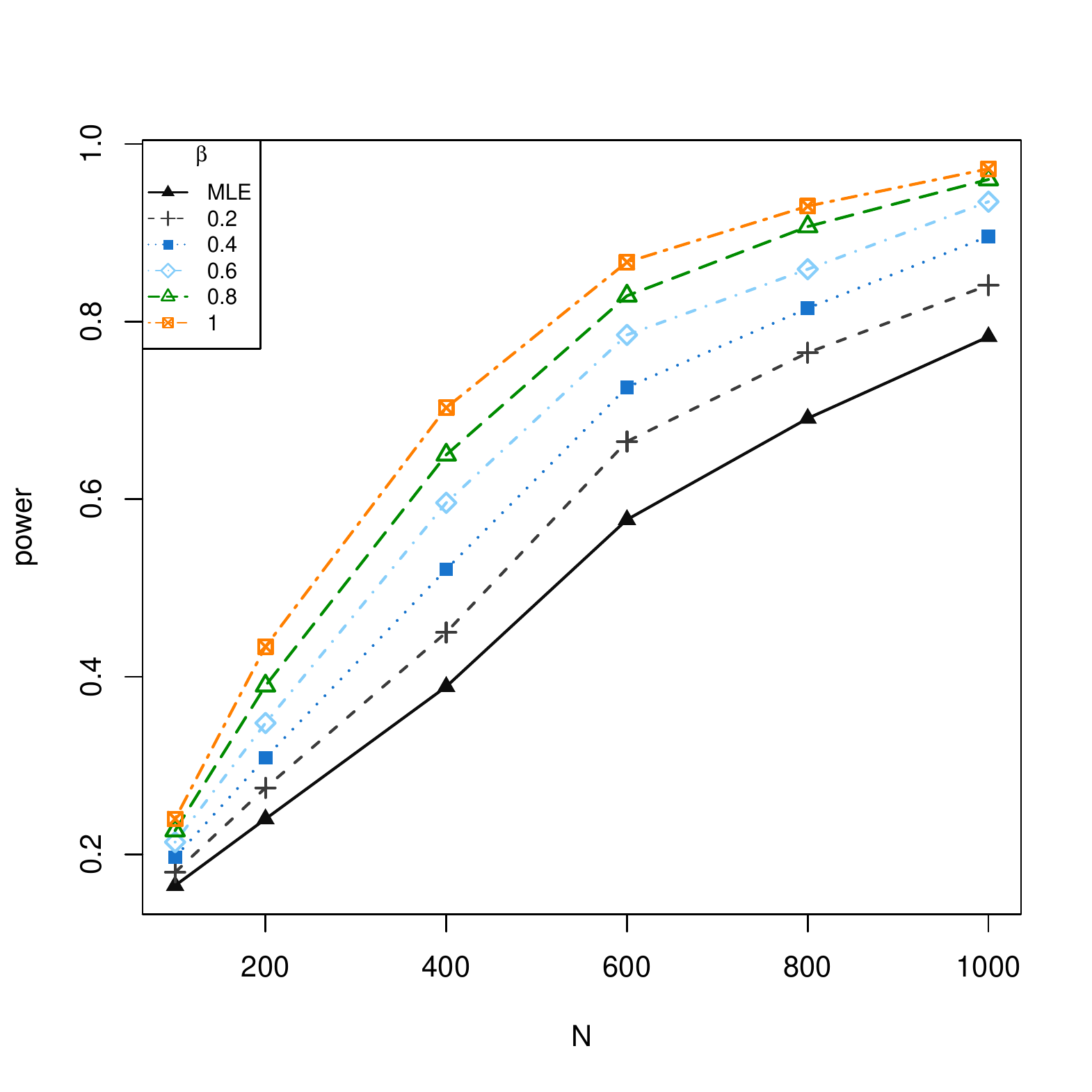}
		\subcaption{$\theta_0$-contaminated cell}
	\end{subfigure}
	\begin{subfigure}{0.5\textwidth}
		\includegraphics[height=7cm, width=8cm]{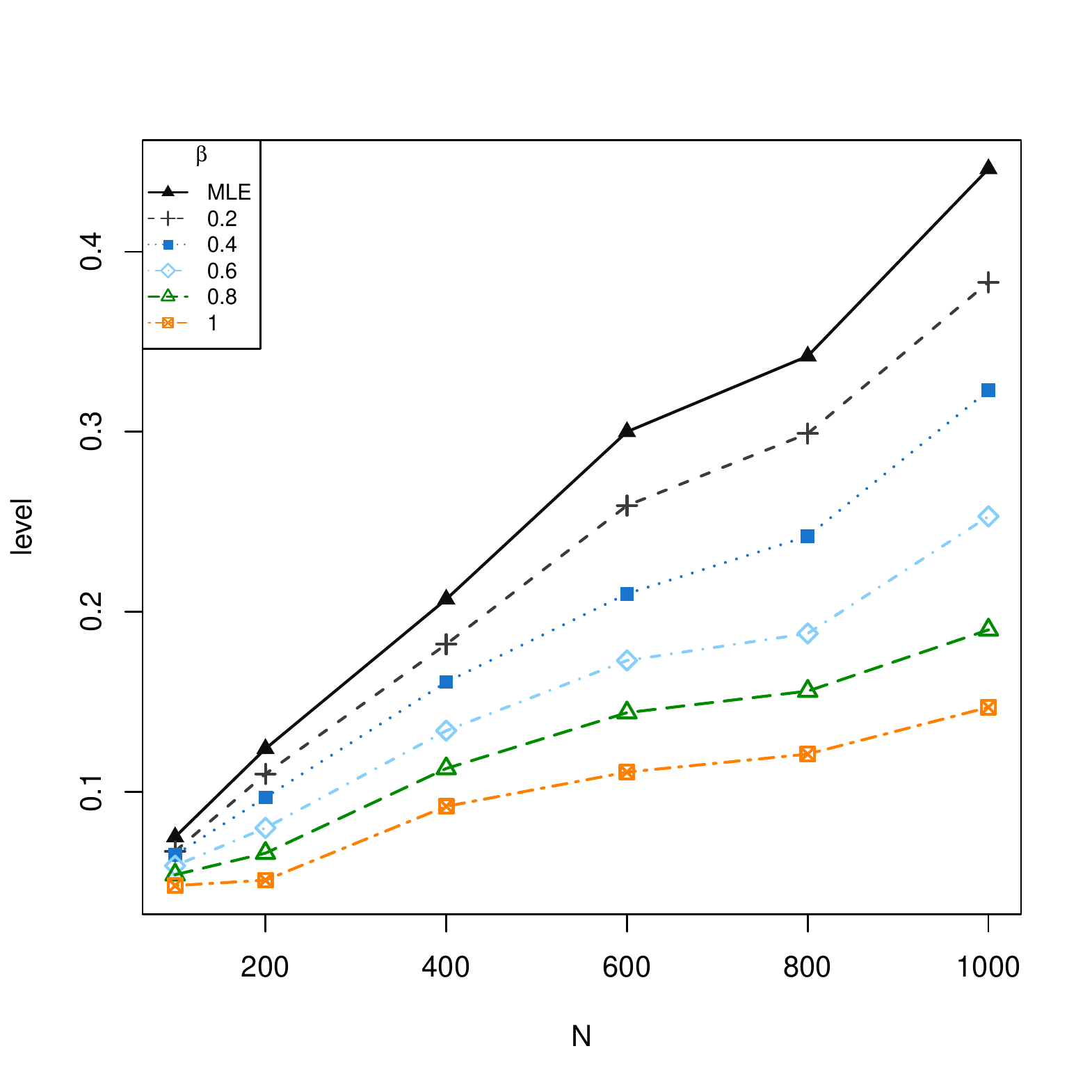}
		\subcaption{$\theta_1$-contaminated cell}
	\end{subfigure}
	\begin{subfigure}{0.5\textwidth}
		\includegraphics[height=7cm, width=8cm]{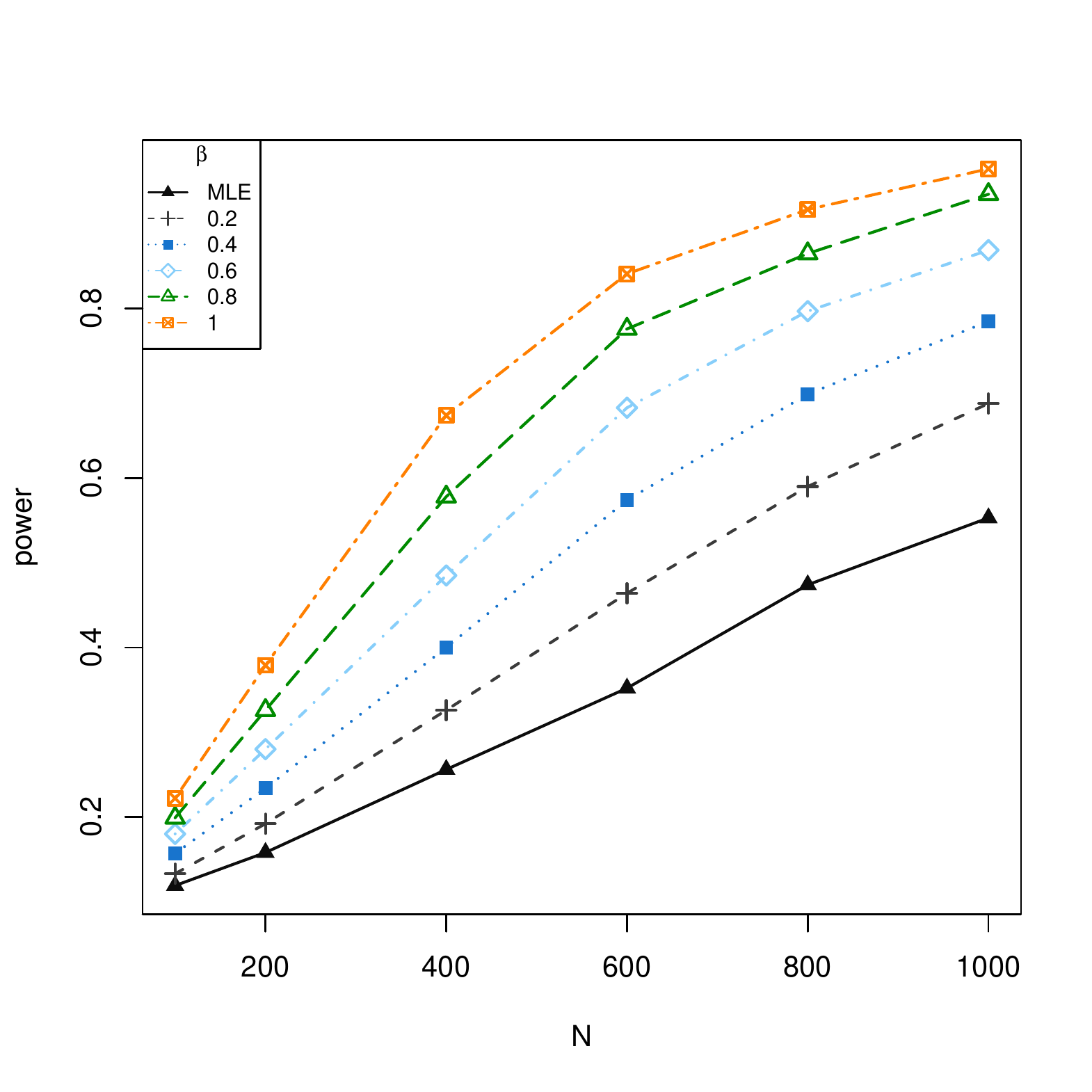}
		\subcaption{$\theta_1$-contaminated cell}
	\end{subfigure}
	\caption{Empirical level and power against sample size with different contaminated scenarios in $R=1000$ replications}
	\label{fig:raolevelandpowerN}
\end{figure}

\section{Conclusions}

In this paper with have developed the restricted minimum MDPDE for one-shot devices data tested under step-stress ALTs. We have derived its asymptotic distribution and analyzed its robustness properties through its IF.
Further, we have defined two families of robust testing procedures based on the restricted DPD estimators, Rao-type test statistics and DPD-based statistics, and we have established the asymptotic distribution for the first one.
Finally, all properties stated theoretically have been illustrated empirically through simulation, showing the certain advantage in terms of robustness of DPD-based inference methods. 

The robustness of the the two proposed families of tests statistics would be interesting to study theoretically through their IF analysis, and furthermore, empirical performance comparison of these two families, together with well-known Wald-type test statistics based on the DPD for one-shot devices tested under step-stress ALT model, will be interesting to be examined in future researches.

\end{document}